\documentclass[12pt]{article}
\usepackage[english]{babel}
\usepackage[utf8]{inputenc}
\usepackage{geometry}
\usepackage{amsmath,amsthm,amssymb,amsfonts}
\usepackage{mathtools}
\usepackage{bbm}
\usepackage{xcolor}
\usepackage{dsfont}
\usepackage{booktabs}
\usepackage{paralist}
\usepackage{tikz}
\usepackage{titlesec}
\usepackage{xcolor}
\usepackage[colorlinks=true,pdfstartview=FitV,linkcolor=blue,citecolor=blue,urlcolor=blue,final]{hyperref}
\numberwithin{equation}{section}
\theoremstyle{plain}
\newtheorem{theorem}{Theorem}[section]
\newtheorem{corollary}[theorem]{Corollary}

\newtheorem{lemma}[theorem]{Lemma}
\newtheorem{remark}[theorem]{Remark}

\newcommand{\ii}{\mathsf{i}}
\newcommand{\T}{\mathbb{S}^1}
\newcommand{\Z}{\mathbb{Z}}
\newcommand{\e}{\mathrm{e}}
\newcommand{\keywords}[1]{%
  \par\noindent\textbf{Keywords:} #1%
}

\newcommand{\msc}[1]{%
  \par\noindent\textbf{2020 Mathematics Subject Classification:} #1%
}

\titleformat{\section}
	{\centering\normalfont\Large}{\thesection.}{0.5em}{}

\titleformat{\subsection}[runin]
	{\normalfont\bfseries}{\thesubsection.}{0.5em}{\addperiod}

\titleformat{\subsubsection}[runin]
	{\normalfont\bfseries}{\thesubsubsection.}{0.5em}{\addperiod}

\titleformat*{\subsubsection}{\normalfont\itshape}
\titleformat*{\paragraph}{\bfseries}
\titleformat*{\subparagraph}{\large\bfseries}

\title{\bf \Large Unimodular Zeros of Selmer's Harmonic Trinomials}

\author{Gerardo Barrera
\thanks{Centro de Análise Matemática, Geometria e Sistemas Dinâmicos, Departamento de Matemática, Instituto Superior Técnico, Universidade de Lisboa. Av. Rovisco Pais 1, 1049-001, Lisboa, Portugal.
{\footnotesize
\href{https://orcid.org/0000-0002-8012-2600}{https://orcid.org/0000-0002-8012-2600}\,
\href{mailto:gerardo.barrera.vargas@tecnico.ulisboa.pt}
{gerardo.barrera.vargas@tecnico.ulisboa.pt}.}}
\and
Waldemar Barrera
\thanks{Facultad de Matem\'aticas, Universidad Aut\'onoma de Yucat\'an.
Anillo Perif\'erico Norte Tablaje CAT 13615, M\'erida, Yucat\'an, M\'exico.
{\footnotesize
\href{https://orcid.org/0000-0001-6885-5556}{https://orcid.org/0000-0001-6885-5556}\,
\href{mailto:bvargas@correo.uady.mx}{bvargas@correo.uady.mx}.}}
\and
Juan Pablo Navarrete
\thanks{Facultad de Matem\'aticas, Universidad Aut\'onoma de Yucat\'an.
Anillo Perif\'erico Norte Tablaje CAT 13615, M\'erida, Yucat\'an, M\'exico.
{\footnotesize
\href{https://orcid.org/0000-0002-3930-4365}{https://orcid.org/0000-0002-3930-4365}\,
\href{mailto:jp.navarrete@correo.uady.mx}{jp.navarrete@correo.uady.mx}.}}
}

\date{}

\begin{document}
\maketitle
\begin{abstract}
Let $n>k$ be positive integers and let
$\varepsilon,\delta\in\{-1,1\}$. In this note, we classify the unimodular zeros of the
harmonic trinomial $H_{\varepsilon,\delta}:\mathbb{C}\to \mathbb{C}$ given by 
\[
H_{\varepsilon,\delta}(z)=z^n+\varepsilon\overline z^{\,k}+\delta,
\]
where $\overline{z}$ is the complex conjugate of $z\in \mathbb{C}$.
Let $g:=\gcd(n,k)$ be the greatest common divisor of $n$ and $k$, and set $n_0:=n/g$ and $k_0:=k/g$. For each choice of
$\varepsilon,\delta\in \{-1,1\}$, we give a necessary and sufficient congruence
involving $n_0$ and $k_0$ to guarantee that $H_{\varepsilon,\delta}$ has unimodular
zeros. More precisely, if the corresponding congruence holds, then
$H_{\varepsilon,\delta}$ has exactly $2g$ unimodular zeros, which are given
explicitly; otherwise, it has no unimodular zeros. The proof uses a reduction to a trinomial equation and elementary divisibility arguments.
\end{abstract}

\keywords{Harmonic Trinomials, Roots of Unity, Selmer's Harmonic Trinomials, Unimodular Zeros, Unit Circle}

\msc{Primary 12D10, 30C15; Secondary 26C10, 30C10.}

\section{Introduction}
Brilleslyper and Schaubroeck \cite{BrilleslyperSchaubroeck} characterize the unimodular zeros (zeros of modulus one) of $z^n+z^k-1$ in terms of arithmetic congruence conditions on the exponents $n$ and $k$. This motivates the corresponding question for harmonic trinomials: how do the exponents and the choices of signs determine whether the trinomial has unimodular zeros?

In this note, for integers $n>k\geq 1$, we study a harmonic analogue of Selmer's trinomials, a family with applications in algebra and number theory~\cite{Ljunggren1960,Selmer1956,Tverberg1960}. More precisely, for $\varepsilon,\delta\in\{-1,1\}$, consider the complex-valued harmonic trinomial
\begin{equation}\label{eq:H}
H_{\varepsilon,\delta}(z)
:=z^n+\varepsilon\overline z^{\,k}+\delta.
\end{equation}
In what follows, let
$g:=\gcd(n,k)$ be the greatest common divisor of $n$ and $k$, and define
\begin{equation}\label{eq:n0k0}
n_0:=n/g\quad \textrm{ and }\quad k_0:=k/g.
\end{equation}
Our goal is to determine when $H_{\varepsilon,\delta}$ has a zero on
the unit circle $\T:=\{u\in \mathbb{C}:|u|=1\}$ and to give all such zeros explicitly.

A problem posed naturally in the theory of \textit{harmonic trinomials} turns out, on the unit circle, to transform into a problem involving \textit{trinomials} and roots of unity. More precisely,
If $z\in \T$, then $\overline z^{\,k}=z^{-k}$, and hence
\begin{equation}\label{eq:polynomial-reduction}
H_{\varepsilon,\delta}(z)=0
\quad\textrm{if and only if}\quad
T(z):=z^{n+k}+\delta z^k+\varepsilon=0.
\end{equation}

Consequently, the problem can be reduced to the study of unimodular roots of the trinomial $T$. For general results on root counts of trinomials and harmonic trinomials,  see \cite{BBNRoots,BBNStabi,Bohl,Botta,BRILLESLYPER,BrilleslyperSchaubroeck2018,Cermak2022} and the references therein; for geometric root configurations, see \cite{BBNEgervary,egervaryoriginal,Gao}; and for the geometry and topology of parameter spaces, see \cite{BCNParameter,BMNAV,TheobaldWolff}.

The main result of this note is the following.
\begin{theorem}[Unimodular zeros]\label{thm:main} 
The harmonic trinomial $H_{\varepsilon,\delta}$ defined in \eqref{eq:H}
has unimodular zeros if and only if the corresponding condition below holds. The possible values of the $g$th powers of its unimodular zeros are listed in the third column.
\begin{center}
\begin{tabular}{@{}clc@{}}
\toprule
$(\varepsilon,\delta)$ & \multicolumn{1}{c}{condition} & values of $z^g$ \\
\midrule
$(1,1)$
  & $n_0\equiv k_0\pmod{3}$
  & $\e^{\pm \ii\, 2\pi /3}$ \\

$(1,-1)$
  & $n_0\equiv k_0\pmod{6}$
  & $\e^{\pm \ii\,\pi /3}$ \\

$(-1,1)$
  & $n_0+2k_0\equiv 0\pmod{6}$
  & $\e^{\pm \ii\,\pi /3}$ \\

$(-1,-1)$
  & $4n_0-k_0\equiv 0\pmod{6}$
  & $\e^{\pm \ii\,\pi /3}$ \\
\bottomrule
\end{tabular}
\end{center}
Moreover, if the condition in the second column holds, then $H_{\varepsilon,\delta}$ has exactly $2g$
distinct unimodular zeros.
\end{theorem}

Theorem~\ref{thm:main} yields the following explicit formula for the unimodular zeros.

\begin{corollary}[Characterization]\label{cor:explicit}
For $\varepsilon,\delta\in\{-1,1\}$, define
\[
 \theta:=\begin{cases}
 2\pi/3,&(\varepsilon,\delta)=(1,1),\\
 \pi/3,&(\varepsilon,\delta)\ne(1,1).
 \end{cases}
\]
Under the corresponding congruence in Theorem~\ref{thm:main},
the unimodular zeros of $H_{\varepsilon,\delta}$ are given by
\begin{equation}\label{eq:explicit-roots}
 z_j^{\pm}
   =\exp\!\left(\ii\frac{(\pm\theta+2\pi j)}{g}\right),
 \quad j=0,1,\ldots,g-1.
\end{equation}
\end{corollary}
\begin{proof}
In all cases covered by Theorem~\ref{thm:main}, the solutions of $z^g=\e^{\pm\ii\theta}$ are
\begin{equation}\label{eq:roots}
z_j^{\pm}
=\exp\left(\ii\frac{(\pm\theta+2\pi j)}{g}\right),
\quad j=0,\ldots,g-1.
\end{equation}
Each sign gives $g$ distinct solutions. The two sets are disjoint because
$\e^{\ii\theta}\ne\e^{-\ii\theta}$. Consequently, there are exactly $2g$ unimodular
zeros.
\end{proof}

We now prove Theorem~\ref{thm:main}. We first reduce the exponents to a coprime pair.
\begin{lemma}[Reduction lemma]\label{lem:reduction}
For $z\in\T$, set $w:=z^g$.
Then
\begin{equation}\label{eq:reduced}
H_{\varepsilon,\delta}(z)=0
\quad\textrm{if and only if}\quad
w^{n_0}+\varepsilon w^{-k_0}+\delta=0.
\end{equation}
Every solution $w\in\T$ of the equation on the right-hand side of \eqref{eq:reduced} has exactly $g$
preimages in $\T$ under the map $z\mapsto z^g$.
\end{lemma}

\begin{proof}
Recall \eqref{eq:n0k0}.
For $z\in\T$, we note that
$z^n=(z^g)^{n_0}=w^{n_0}$ and
$\overline z^{\,k}=z^{-k}=(z^g)^{-k_0}=w^{-k_0}$.
Hence, $H_{\varepsilon,\delta}(z)=
w^{n_0}+\varepsilon w^{-k_0}+\delta$,
which proves \eqref{eq:reduced}.

Now, let $w\in\T$ be a solution of the equation on the right-hand side of~\eqref{eq:reduced}. Write $w=\e^{\ii t}$ for some $t\in\mathbb{R}$. The solutions of $z^g=w$ are given by
\[
z_j:=
\exp\left(\ii\frac{(t+2\pi j)}{g}\right),
\qquad
j=0,\ldots,g-1.
\]
These $g$ points are all different and belong to $\T$. Therefore, every solution
$w\in\T$ of the reduced equation on the right-hand side of \eqref{eq:reduced} has exactly $g$ preimages in $\T$.
\end{proof}

Setting $u=w^{n_0}$ and $v=w^{-k_0}$ reduces the right-hand side of \eqref{eq:reduced} to the four cases in the next lemma.

\begin{lemma}\label{lem:geometry}
Let $u,v\in\T$ and $\zeta:=\e^{2\pi\ii/3}$.
Then the following statements hold.
\begin{align*}
\textrm{(a)}\quad u+v+1=0
&\quad\Longleftrightarrow\quad
(u,v)=(\zeta,\zeta^2)\ \text{or}\ (\zeta^2,\zeta),\\
\textrm{(b)}\quad u+v-1=0
&\quad\Longleftrightarrow\quad
(u,v)=(-\zeta,-\zeta^2)\ \text{or}\ (-\zeta^2,-\zeta),\\
\textrm{(c)}\quad u-v+1=0
&\quad\Longleftrightarrow\quad
(u,v)=(\zeta,-\zeta^2)\ \text{or}\ (\zeta^2,-\zeta),\\
\textrm{(d)}\quad u-v-1=0
&\quad\Longleftrightarrow\quad
(u,v)=(-\zeta,\zeta^2)\ \text{or}\ (-\zeta^2,\zeta).
\end{align*}
\end{lemma}
\begin{proof}
We prove (a); (b), (c), and (d) follow by applying (a) to $(-u,-v)$, $(u,-v)$, and $(-u,v)$, respectively.
Suppose $u+v+1=0$. Then $v=-1-u$, and
\[
1=|v|^2=|1+u|^2=2+2\operatorname{Re}u.
\]
Therefore, $\operatorname{Re}u=-1/2$. Write $u=x+\ii y$ with $x,y\in\mathbb R$ and note that
$y=\pm\sqrt3/2$. 
Hence, $u=\zeta$ or $u=\zeta^2$. Since $\zeta^3=1$ and $\zeta\ne1$, we have $1+\zeta+\zeta^2=0$. Therefore, $v=-1-u$ is the other member of $\{\zeta,\zeta^2\}$. Conversely, both ordered pairs have unit-modulus entries and satisfy $u+v+1=0$.
\end{proof}

\begin{remark}[Orders of roots of unity]\label{rem:primitive}
Recall that $\zeta:=\e^{2\pi\ii/3}$.
Note that the cube roots of unity are
$1,\zeta,\zeta^2$ and have orders $1,3,3$, respectively.
Moreover, the sixth roots of unity are $1,-\zeta^2,\zeta,-1,\zeta^2,-\zeta$,
of orders $1,6,3,2,3,6$, respectively. 
Indeed, 
for $m\in\Z$, $(\zeta^2)^m=1$ exactly when $3\mid2m$, or, equivalently, when $3\mid m$.
The other orders follow in the same way from divisibility.
Consequently, the only roots of unity of order six are
$-\zeta$ and $-\zeta^2$.
\end{remark}

We can now complete the proof of Theorem~\ref{thm:main}.
\begin{proof}[Proof of Theorem~\ref{thm:main}]
Write $w=z^g$. By Lemma~\ref{lem:reduction}, it is enough to solve
\begin{equation}\label{eq:weqn}
w^{n_0}+\varepsilon w^{-k_0}+\delta=0,
\quad w\in\T.
\end{equation}
Suppose that $(\varepsilon,\delta)=(1,1)$.
The proofs of the other cases proceed similarly, using parts (b), (c), and (d) of 
Lemma~\ref{lem:geometry} and Remark~\ref{rem:primitive}.
Then \eqref{eq:weqn} becomes
$w^{n_0}+w^{-k_0}+1=0$.
By part (a) of Lemma~\ref{lem:geometry}, we obtain
$\{w^{n_0},w^{-k_0}\}=\{\zeta,\zeta^2\}$,
which implies $w^{3n_0}=w^{3k_0}=1$. Since $\gcd(n_0,k_0)=1$, B\'ezout's identity
gives $w^3=1$. Note that $w=1$ does not satisfy $w^{n_0}+w^{-k_0}+1=0$.
Hence, $w\in\{\zeta,\zeta^2\}$.
Assume $w=\zeta$. Using the division Euclides algorithm, write $n_0=3p_0+r_0$ and $k_0=3p'_0+r'_0$ for some $p_0,p'_0\in \mathbb{N}\cup\{0\}$ and $r_0,r'_0\in \{0,1,2\}$. We observe that
\begin{equation}
\begin{split}
0=w^{n_0}+w^{-k_0}+1&=\zeta^{n_0}+\zeta^{-k_0}+1=
\zeta^{3p_0+r_0}+\zeta^{-3p'_0-r'_0}+1=\zeta^{r_0}+\zeta^{-r'_0}+1.
\end{split}
\end{equation}
Since $|\zeta|=1$, this implies $r_0,r'_0\in \{1,2\}$.
If $r_0=r'_0$, the statement is complete.
Assume without loss of generality that $r_0=1$ and $r'_0=2$. Note that
\begin{equation}
\begin{split}
0=\zeta+\zeta^{-2}+1=\zeta+\zeta+1,
\end{split}
\end{equation}
which yields a contradiction because $|\zeta|=1$.

The proof of the case $w=\zeta^2$ follows similarly.
\end{proof}

\section*{\textbf{Statements and declarations}}

\noindent
\textbf{Acknowledgments.} Gerardo Barrera would like to express his gratitude to the Center for Mathematical Analysis, Geometry and Dynamical Systems CAMGSD and Instituto Superior T\'ecnico (IST) Lisbon, Portugal
for all the facilities used along the realization of this work.

\noindent
\textbf{Funding.}
The research of Gerardo Barrera is funded by FCT/Portugal and the Recovery and Resilience Plan (PRR) through projects UID/04459/2025 (DOI: 10.54499/UID/04459/2025) and UID/PRR/04459/2025 (DOI:10.54499/UID/PRR/04459/2025).
The research of Waldemar Barrera is partially supported by SECIHTI-SNI 35822.
The research of Juan Pablo Navarrete is partially supported by SECIHTI-SNI 35874.

\noindent
\textbf{Availability of data and material.}
Data sharing not applicable to this article as no data-sets were generated or analyzed during the current study.

\noindent
\textbf{Conflict of interests.} The authors declare that they have no conflict of interest.

\noindent
\textbf{Ethical approval.} Not applicable.

\noindent
\textbf{Authors' contributions.}
All authors have contributed equally to the manuscript.

\end{document}